\documentclass{amsart}
\usepackage{amsmath}
\usepackage{amssymb}
\usepackage{amsthm}

\DeclareMathOperator{\rank}{rank}

\DeclareMathOperator{\Imaginary}{Im}
\DeclareMathOperator{\Real}{Re}

\newcommand{\lp}{\langle}
\newcommand{\rp}{\rangle}
\newcommand{\lv}{\lvert}
\newcommand{\rv}{\rvert}
\newcommand{\lV}{\lVert}
\newcommand{\rV}{\rVert}

\newcommand{\mP}{\mathcal{P}}

\newcommand{\bC}{\mathbb{C}}

\def\sideremark#1{\ifvmode\leavevmode\fi\vadjust{\vbox to0pt{\vss
 \hbox to 0pt{\hskip\hsize\hskip1em
 \vbox{\hsize3cm\tiny\raggedright\pretolerance10000
 \noindent #1\hfill}\hss}\vbox to8pt{\vfil}\vss}}}

\newcommand{\comment}[1]{}

\newtheorem{thm}{Theorem}[section]

\theoremstyle{definition}

\theoremstyle{remark}
\newtheorem{remark}[thm]{Remark}

\numberwithin{equation}{section}
\usepackage[all]{xy}

\begin{document}

\title{A Remark on the Kernel of the CR Paneitz operator}
\author{Jeffrey S.\ Case}
\address{Department of Mathematics \\ Princeton University \\ Princeton, NJ 08540}
\email{jscase@math.princeton.edu}
\author{Sagun Chanillo}
\thanks{SC was partially supported by NSF Grant No.\ DMS-1201474}
\address{Department of Mathematics \\ Rutgers University \\ 110 Frelinghuysen Rd., Piscataway, NJ 08854}
\email{chanillo@math.rutgers.edu}
\author{Paul Yang}
\thanks{PY was partially supported by NSF Grant No.\ DMS-1104536}
\address{Department of Mathematics \\ Princeton University \\ Princeton, NJ 08540}
\email{yang@math.princeton.edu}
% \date{}
% \keywords{}
% \subjclass[2000]{Primary ?; Secondary ?}
 \begin{abstract}
For CR structures in dimension three, the CR pluriharmonic functions are characterized
by the vanishing of a third order operator. This third order operator, after composition with the divergence operator, gives the fourth order analogue of the Paneitz operator. In this short note, we give criteria under which the kernel of the CR Paneitz operator contains a supplementary space to the CR pluriharmonic functions.
\end{abstract}
\maketitle

\section{Introduction}
\label{sec:intro}

It is well-known that for a three-dimensional CR manifold, the kernel of the CR Paneitz operator contains the CR pluriharmonic functions.  It is an open question whether there are examples of CR manifolds for which the kernel of the CR Paneitz operator is exactly the CR pluriharmonic functions.  Graham and Lee showed~\cite{GrahamLee1988} that this is true for compact CR manifolds admitting a torsion free contact form.  More recently, the authors showed~\cite{CaseChanilloYang2014} that this is true for real analytic families of CR manifolds for which the space of CR pluriharmonic functions is stable, assuming positivity of certain natural geometric invariants.  In this short note, we give conditions, related to the instability of the space of CR pluriharmonic functions along deformations, which guarantee that the kernel of the CR Paneitz operator contains functions which are not CR pluriharmonic; see Theorem~\ref{thm:2.1} and Theorem~\ref{thm:2.36}.

Throughout this paper, we use the notation and terminology in \cite{Lee1988} unless
otherwise specified.
Let $(M,J,\theta)$ be a closed three-dimensional pseudohermitian manifold,
where $\theta$ is a contact form and $J$ is a CR structure compatible with the
contact bundle $\xi=\ker\theta$.
The CR structure $J$ decomposes $\xi\otimes\bC$ into the $+i$- and $-i$-eigenspaces of $J$, denoted $T_{1,0}$ and $T_{0,1}$, respectively.
The Levi form $\left\langle\  ,\ \right\rangle_{L_\theta}$ is the Hermitian form on
$T_{1,0}$ defined by
$\left\langle Z,W\right\rangle_{L_\theta}=
-i\left\langle d\theta,Z\wedge\overline{W}\right\rangle$.
We can extend $\left\langle\  ,\ \right\rangle_{L_\theta}$ to $T_{0,1}$ by defining
$\left\langle\overline{Z} ,\overline{W}\right\rangle_{L_\theta}=
\overline{\left\langle Z,W\right\rangle}_{L_\theta}$ for all $Z,W\in T_{1,0}$.
The Levi form induces naturally a Hermitian form on the dual bundle of $T_{1,0}$,
denoted by $\left\langle\  ,\ \right\rangle_{L_\theta^*}$,
and hence on all the induced tensor bundles.  By integrating the Hermitian form
(when acting on sections) over $M$ with respect to the volume form $dV=\theta\wedge d\theta$,
we get an inner product on the space of sections of each tensor bundle.
We denote this inner product by $\left\langle \ ,\ \right\rangle$. For example
\begin{equation}\label{21}
\left\langle\varphi ,\psi\right\rangle=\int_{M}\varphi\bar{\psi}\ dV,
\end{equation}
for functions $\varphi$ and $\psi$.

The Reeb vector field $T$ is the unique vector field such that $\theta(T)=1$ and $\theta(T,\cdot)=0$.  Let $Z_1$ be a local frame of $T_{1,0}$ and consider the frame $\left\{T,Z_1,Z_{\bar 1}\right\}$ of $TM\otimes\bC$. Then $\left\{\theta,\theta^1,\theta^{\bar 1}\right\}$,
the coframe dual to $\left\{T,Z_1,Z_{\bar{1}}\right\}$, satisfies
\begin{equation}\label{22}
d\theta=ih_{1\bar 1}\theta^1\wedge\theta^{\bar 1}
\end{equation}
for some positive function $h_{1\bar 1}$. We can always choose $Z_1$
such that $h_{1\bar 1}=1$; hence, throughout this paper, we assume
$h_{1\bar 1}=1$

The pseudohermitian connection of $(J,\theta)$ is 
$\nabla$ on $TM\otimes\bC$ (and extended to tensors) given in terms of a local
frame $Z_1\in T_{1,0}$ by
\begin{equation*}
\nabla Z_1=\theta_1{}^1\otimes Z_1,\quad
\nabla Z_{\bar{1}}=\theta_{\bar{1}}{}^{\bar{1}}\otimes Z_{\bar{1}},\quad
\nabla T=0,
\end{equation*}
where $\theta_1{}^1$ is the $1$-form uniquely determined by the equations
\begin{align*}
d\theta^1&=\theta^1\wedge\theta_1{}^1+\theta\wedge\tau^1 , \\
\tau^1&\equiv 0\mod\theta^{\bar 1} , \\
0&=\theta_1{}^1+\theta_{\bar{1}}{}^{\bar 1} .
\end{align*}
$\theta_{1}{}^{1}$ and $\tau^1$ are called the connection form and the pseudohermitian
torsion, respectively.
Put $\tau^1=A^1{}_{\bar 1}\theta^{\bar 1}$.
The structure equation for the pseudohermitian connection is
\begin{equation*}
d\theta_1{}^1=R\theta^1\wedge\theta^{\bar 1}
+2i\Imaginary (A^{\bar{1}}{}_{1,\bar{1}}\theta^1\wedge\theta),
\end{equation*}
where $R$ is the Tanaka-Webster curvature.

We denote components of covariant derivatives with indices preceded by a comma;
thus we write $A^{\bar{1}}{}_{1,\bar{1}}\theta^1\wedge\theta$.
The indices $\{0, 1, \bar{1}\}$ indicate derivatives with respect to $\{T, Z_1, Z_{\bar{1}}\}$.
For derivatives of a scalar function, we omit the comma;
for example, given a smooth function $\varphi$, we write $\varphi_{1}=Z_1\varphi$ and $\varphi_{1\bar{1}}=
Z_{\bar{1}}Z_1\varphi-\theta_1^1(Z_{\bar{1}})Z_1\varphi$ and $\varphi_{0}=T\varphi$.

Next we recall several natural differential operators occurring in this paper.
For a smooth function $\varphi$, the Cauchy--Riemann operator $\partial_{b}$ is
defined locally by
\[\partial_{b}\varphi=\varphi_{1}\theta^{1} . \]
We write $\bar{\partial}_{b}$ for the conjugate of
$\partial_{b}$. A function $\varphi$ is called CR holomorphic if
$\bar{\partial}_{b}\varphi=0$. The divergence operator $\delta_b$
takes $(1,0)$-forms to functions by
$\delta_b(\sigma_{1}\theta^{1})=\sigma_{1,}{}^{1}$; similarly,
$\bar{\delta}_b(\sigma_{\bar 1}\theta^{\bar 1})=\sigma_{\bar
1,}{}^{\bar 1}$.

The Kohn Laplacian on functions is
\[\Box_{b}=2\bar{\partial}_{b}^{*}\bar{\partial_{b}} . \]
The sublaplacian is the operator $\Delta_b=\Real\Box_b$.  The CR conformal Laplacian acts on functions $\varphi$ by $L\varphi = -\Delta_b\varphi + (1/4)R\varphi$.

The operator $P_3$ defined on functions $\varphi$ by $P_3\varphi=(\varphi_{\bar{1}}{}^{\bar{1}}{}_1+iA_{11}\varphi^1)\theta^1$ characterizes the space $\mP$ of CR pluriharmonic functions as $\mP=\ker P_3$ (see \cite{Lee1988}).  The CR Paneitz operator $P_4$ is defined by
\begin{equation}
\label{eqn:P4_defn}
P_4\varphi=\delta_b(P\varphi).
\end{equation}
Define $Q$ by $Q\varphi=2i(A^{11}\varphi_{1})_{,1}$.  Using the commutation relation $[\Box_b,\overline{\Box}_b]=4i\Imaginary Q$, we see that
\begin{equation*}
\begin{split}
P_4\varphi&=\frac{1}{4}(\Box_{b}\overline{\Box}_{b}-2Q)\varphi\\
&=\frac{1}{8}\big((\overline{\Box}_{b}\Box_{b}+\Box_{b}\overline{\Box}_{b})\varphi+8\Imaginary(A^{11}\varphi_{1})_{1}\big).
\end{split}
\end{equation*}
Hence $P_4$ is a real and symmetric operator. Note that the leading order term of $P_4$ makes it a fourth order hyperbolic operator, thus it is remarkable that it still displays properties of a subelliptic operator.  It follows from~\eqref{eqn:P4_defn} that the CR pluriharmonic functions are contained in the kernel of the CR Paneitz operator.  It is natural to ask whether there is anything else in the kernel of $P_4$.  Hsiao recently showed~\cite{Hsiao2014} that, under the assumption that the structure $(M^3,\theta,J)$ is embedded, there can be
at most a finite dimensional supplementary of the CR pluriharmonics in the kernel of $P_4$. More
precisely, for any embedded structure $(M,J,\theta)$, we have
\begin{equation}
\label{eqn:supplementary}
\ker P_4 = \mP \oplus W
\end{equation}
for $W$ a finite dimensional space, henceforth referred to as the supplementary space.  An elementary proof of this fact is presented as Lemma~2.2 in~\cite{CaseChanilloYang2014}.  We say that the supplementary space exists if $W\not=\{0\}$ in~\eqref{eqn:supplementary}.

In previous work~\cite{CaseChanilloYang2014} on the kernel of the CR Paneitz operator, the authors showed that the supplementary space does not exist under natural conditions on an analytic family of CR manifolds.

\begin{thm}
\label{thm:main_thm}
Let $(M^3,J^t,\theta)$ be a family of embedded CR manifolds for $t\in[-1,1]$ with the following properties.
\begin{enumerate}
\item $J^t$ is real analytic in the deformation parameter $t$.
\item The Szeg{\H o} projectors $S^t\colon F^{2,0}\to(\ker \bar\partial_b^t\subset F^{2,0})$ vary continuously in the deformation parameter $t$ (see~\cite{GarfieldLee1998} for a definition of $F^{2,0}$).
\item For the structure $J^0$ we have $\lp P_4\psi,\psi\rp\geq0$ for all functions $\psi$ and $\ker P_4^0=\mP^0$, the space of CR pluriharmonic functions with respect to $J^0$.
\item There is a uniform constant $c>0$ such that
\begin{equation}
\label{eqn:uniform_curv}
\inf_{t\in[-1,1]} \min_M R^t \geq c > 0 .
\end{equation}
\item The CR pluriharmonic functions are stable for the family $(M^3,J^t,\theta)$.
\end{enumerate}
Then $P_4^t\geq0$ and $\ker P_4^t=\mP^t$ for all $t\in[-1,1]$.
\end{thm}

\begin{remark}
Let $Y[J]$ denote the CR Yamabe constant of a compact CR manifold; that is,
\[ Y[J] := \inf_\theta \left\{ \int R\,\theta\wedge d\theta \colon \int \theta\wedge d\theta 
= 1 \right\} \]
with the infimum taken across all contact forms on $(M^3,\xi)$.  
The assumption \eqref{eqn:uniform_curv} can be replaced by the stronger assumption $\inf_t Y[J^t] \geq c > 0$, in which case the assumptions of Theorem~\ref{thm:main_thm} are all CR invariant.
\end{remark}

In view of \cite{ChanilloChiuYang2013} the hypotheses of theorem (1.2) are satisfied by the family
of ellipsoids in $\bC^2$.

Theorem~\ref{thm:main_thm} shows that the stability of the CR pluriharmonic functions plays a role in preventing the existence of the supplementary space.  If one wishes to use deformations to exhibit examples of CR manifolds for which the supplementary space exists, one should thus look at unstable families.  In the next section, we give two conditions which guarantee the existence of the supplementary space.

%\section*{Acknowledgments}

\section{The Supplementary Space and Instability}
\label{sec:baire}

Using estimates from~\cite{CaseChanilloYang2014}, we prove two results about the supplementary space for families of embedded CR manifolds without imposing a stability assumption.

First, using the Baire Category Theorem, we show that generically the supplementary space exists.

\begin{thm}
\label{thm:2.1}
Let $(M^3,J^t,\theta)$ be a family of embedded CR manifolds with $t\in[-1,1]$.  Assume that the Szeg{\H o} projector $S^t\colon F^{2,0}\to(\ker \bar\partial_b^t\subset F^{2,0})$ varies real analytically in the deformation variable $t$ (see~\cite{GarfieldLee1998} for a definition of $F^{2,0}$).  Then
\begin{enumerate}
\item $n_0 := \displaystyle\sup_{t\in[-1,1]} \dim W^t < \infty$.
\item The set
\[ F := \left\{ t\in[-1,1] \colon \dim W^t \leq n_0 - 1 \right\} \]
is a closed set with no accumulation points.  In particular, $F$ has no interior points and hence $F$ is of the first category.
\item The set
\[ E := \left\{ t\in[-1,1] \colon \dim W^t = n_0 \right\} \]
has nonempty interior.
\end{enumerate}
\end{thm}

\begin{remark}
The theorem above states that for generic values $t$ of the deformation parameter, $\dim W^t=n_0$.  Since $n_0>0$ if there exists a $t_0\in[-1,1]$ with $W^{t_0}\not=\{0\}$, the supplementary space exists for a generic value of $t$ if it exists for some $t_0$.  Moreover, if $n_0>0$, then $\dim W^t=0$ for a thin set $F\subset[-1,1]$ of the first category.
\end{remark}

\begin{proof}

To prove the theorem we apply \cite[Lemma~2.28]{BurnsEpstein1990}.  The statement of 
\cite[Lemma~2.28]{BurnsEpstein1990} is for deformation parameters that are complex analytic, 
but we here apply it to the finite rank operators $A^t$ constructed in 
\cite[Lemma~2.9]{CaseChanilloYang2014} 
which depend the real variable $t$.  The assumption that $J^t$ and $S^t$ depend 
real-analytically on $t$ yields that the function $h(t)=\lp A^t\phi,\psi\rp$ from the 
proof of \cite[Lemma~2.9]{CaseChanilloYang2014} is real analytic in $t$.  
Thus the zeros of $h(t)$ are isolated and we conclude exactly as in 
\cite[Lemma~2.28]{BurnsEpstein1990} that
\[ \sup_{t\in[-1,1]} \rank A^t = \sup_{t\in[-1,1]} \dim W^t < \infty , \]
where we use \cite[Lemma~2.8]{CaseChanilloYang2014} to establish the first equality.  
This proves the first 
claim.

Next, \cite[Corollary 2.11]{CaseChanilloYang2014} implies that $\dim W^t$ is lower 
semi-continuous, and 
thus the set $F$ is closed.  By the conclusion of \cite[Lemma~2.28]{BurnsEpstein1990}, 
the set
\[ \left\{ t\in[-1,1] \colon \rank A^t < n_0 \right\} 
= \left\{ t\in[-1,1] \colon \rank A^t\leq n_0-1 \right\} \]
has no accumulation points.  This and \cite[Lemma 2.8]{CaseChanilloYang2014} yield the 
second claim.

Finally, note that
\begin{equation}
\label{eqn:sets}
\left\{ t\in[-1,1] \colon \dim W^t > n_0-1\right\} = 
\left\{ t\in[-1,1] \colon \rank A^t > n_0 - 1 \right\}
\end{equation}
cannot be empty; if it were, this would contradict the definition of $n_0$.  
Since $\dim W^t$ is lower semi-continuous, the sets~\eqref{eqn:sets} are open and nonempty, 
yielding the last claim.
\end{proof}

Second, we show that, under an assumption on the rate of vanishing of the 
first eigenvalue of the CR Paneitz operator for a 
family of CR structures, the loss of stability of the CR pluriharmonic functions 
implies the existence of the supplementary space.  To make this precise, we list our 
assumptions.

Let $(M^3,J^t,\theta)=:M^t$ be a family of embedded CR manifolds for which $J^t$ is $C^6$ in the deformation parameter $t$ for some interval $\lv t-t_0\rv < \mu$ with $\mu>0$.  Suppose that there is a constant $c>0$ independent of $t$, such that:
\begin{enumerate}
\item For any $t\not=t_0$ and any $f\perp\ker P_4^t$, it holds that
\begin{subequations}
\label{eqn:2.32}
\begin{equation}
\label{eqn:2.32a}
\lv t-t_0\rv \eta\left(\lv t-t_0\rv\right) \lV f\rV_2 \leq \lV P_4^t f\rV_2,
\end{equation}
where $\eta(s) \rightarrow \infty$ as $s$ tends to zero.
\item For any $f\perp\ker P_4^{t_0}$, it holds that
\begin{equation}
\label{eqn:2.32b}
c \lV f\rV_2 \leq \lV P_4^{t_0}f \rV_2 .
\end{equation}
\end{subequations}
\end{enumerate}
Note that the inequality~\eqref{eqn:2.32b} was observed by Cao and Chang~\cite{CaoChang2007} for a fixed embedded CR manifold.  Together, the assumptions~\eqref{eqn:2.32a} and~\eqref{eqn:2.32b} imply that the lowest nonzero absolute value of the eigenvalues of the CR Paneitz operator $P_4^t$ jumps up as $t\to t_0$.

Next, we assume there is a family of diffeomorphisms $\Phi^t\colon M^t \to M^0:=M^{t_0}$ which is $C^6$ in the deformation parameter $t$ and is such that $\Phi^0$ is the identity map.

Lastly, we assume that the CR pluriharmonic functions are unstable at $t_0$.  More precisely, we assume that there is a CR pluriharmonic function $f_0\in C^5$ for the structure $M^0$ and constants $\varepsilon>0$ and $0<\delta<\mu$ such that for any $t$ with $\lv t-t_0\rv<\delta$ and any CR pluriharmonic function $\psi\in\mP^t$, it holds that
\begin{equation}
\label{eqn:2.34}
\lV f_0-\psi\rV_2 \geq \varepsilon .
\end{equation}

%\begin{remark}
%In view of the proof of Lemma~\ref{lem:2.1}, it seems that a hypothesis of the form
%\[ \lambda_1\left(\Box_b^t\right) \geq c\lv t-t_0\rv^\gamma \]
%for some $\gamma<1/2$ would lead to~\eqref{eqn:2.32}.
%\end{remark}

\begin{thm}
\label{thm:2.36}
Assume $(M^3,J^t,\theta)$ is as described above.  Then for all $t\not=t_0$ with $\lv t-t_0\rv<\delta$ and $\delta$ sufficiently small, the supplementary space $W^t$ exists.
\end{thm}

\begin{proof}

Set $u_t=f_0\circ\Phi^t$.  By assumption, $u_t\in C^5$.  We claim that $u_t$ is not perpendicular to $\ker P_4^t$.  If instead $u_t\perp\ker P_4^t$, then~\eqref{eqn:2.32a} gives
\begin{equation}
\label{eqn:2.38}
\lv t-t_0\rv \eta\left(\lv t-t_0\rv\right) \lV u_t\rV_2 \leq \lV P_4^t u_t\rV_2 .
\end{equation}
On the other hand, using the fact that $f_0\in\mP^{t_0}$, we have that
\[ \lV P_4^tu_t\rV_2 \leq \left\lV\left(P_4^t-P_4^{t_0}\right)u_t\right\rV_2 + \left\lV P_4^{t_0}\left(u_t-f_0\right)\right\rV_2 . \]
Applying~\eqref{eqn:2.38} and the definitions of $\Phi^t$ and $u_t$ thus yields
\[ \lv t-t_0\rv \eta\left(\lv t-t_0\rv\right) \lV u_t\rV_2 \leq c\lv t-t_0\rv . \]
Recalling that $\eta\left(\lv t-t_0\rv\right) \rightarrow \infty$ as $t\to t_0$, we have that $\lV u_0\rV_2=\lV f_0\rV_2=0$, a contradiction.

Now assume that the projection of $u_t$ onto $\ker P_4^t$ is a CR pluriharmonic function $\psi_t\in\mP^t$.  We may write
\begin{equation}
\label{eqn:2.39}
u_t = v_t + \psi_t
\end{equation}
with $v_t\perp\ker P_4^t$.  Since $\psi_t$ is CR pluriharmonic,
\begin{equation}
\label{eqn:2.40}
P_4^tu_t = P_4^tv_t .
\end{equation}
On the other hand, by using the fact that $f_0\in\mP^{t_0}$ as above, we find that
\begin{equation}
\label{eqn:2.41}
\lV P_4^tu_t\rV_2 \leq c\lv t-t_0\rv .
\end{equation}
Since $v_t\perp\ker P_4^t$ we may apply~\eqref{eqn:2.32} together with~\eqref{eqn:2.40} and~\eqref{eqn:2.41} to find that
\[ \lv t-t_0\rv \eta\left(\lv t-t_0\rv\right) \lV v_t\rV_2 \leq \lV P_4^tv_t\rV_2 \leq c\lv t-t_0\rv . \]
In particular,
\begin{equation}
\label{eqn:2.42}
\lV v_t\rV_2 < \varepsilon
\end{equation}
for $\lv t-t_0\rv<\delta$.  Now,
\[ f_0-\psi_t = f_0-u_t+v_t . \]
Hence
\begin{equation}
\label{eqn:2.43}
\lV f_0-\psi_t\rV_2 \leq \lV f_0 - u_t\rV_2 + \lV v_t\rV_2 .
\end{equation}
By~\eqref{eqn:2.42} and the construction of $u_t$, we conclude that $\lV f_0-\psi_t\rV_2=o(1)$ for $t\to t_0$.  This contradicts the stability assumption~\eqref{eqn:2.34}.  Hence the projection of $u_t$ onto $\ker P_4^t$ cannot be a CR pluriharmonic function, whence $W^t\not=\{0\}$.
\end{proof}

\bibliographystyle{abbrv}

\newcommand{\noopsort}[1]{}

% \bibliography{bib}
\end{document}